\documentclass{amsart}
\usepackage{amsmath, amsthm, amsfonts, amssymb}
\usepackage{mathrsfs,graphicx}
\usepackage{bbm}
\usepackage{dsfont}
\usepackage{ifthen}
\usepackage{a4}
\usepackage{hyperref}
\usepackage{enumerate}

\numberwithin{equation}{section}
\newtheorem{thm}{Theorem}[section]
\newtheorem{cor}[thm]{Corollary}
\newtheorem{prop}[thm]{Proposition}
\newtheorem{lemma}[thm]{Lemma}

\theoremstyle{remark}
\newtheorem{rem}[thm]{Remark}

\theoremstyle{definition}
\newtheorem{defn}[thm]{Definition}

\newcommand{\coloneqq}{\mathrel{\mathop:}=}
\newcommand{\eqqcolon}{=\mathrel{\mathop:}}
\newcommand{\applied}[2]{\langle #1,#2\rangle}

\newcommand{\dx}{\mathrm{d}}

\renewcommand{\phi}{\varphi}

\newcommand{\eps}{\varepsilon}

\newcommand{\R}{\mathds{R}}

\newcommand{\N}{\mathds{N}}

\DeclareMathOperator{\supp}{supp}

\newcommand{\fix}{\mathrm{fix}}

\newcommand{\cD}{\mathscr{D}}
\newcommand{\cL}{\mathscr{L}}
\newcommand{\cM}{\mathscr{M}}
\newcommand{\cC}{\mathscr{C}}
\newcommand{\cT}{\mathscr{T}}
\newcommand{\cS}{\mathscr{S}}
\newcommand{\cB}{\mathscr{B}}

\newcommand{\cA}{\mathscr{A}}

\begin{document}

\title[A Tauberian Theorem for Strong Feller Semigroups]{A Tauberian Theorem\\for Strong Feller Semigroups}
\author{Moritz Gerlach}

\begin{abstract}
	We prove that a weakly ergodic, eventually strong Feller semigroup on the space of measures
	on a Polish space converges strongly to a projection onto its fixed space.
\end{abstract}

\maketitle

\section{Introduction}

In the study of Markov processes one is interested in semigroups of operators on the space of measures that describe the evolution
of distributions. 
Of specific importance is the question under which conditions such a semigroup is stable in the sense that for every initial distribution 
the process converges to an invariant measure as time goes to infinity. 
A celebrated theorem by Doob asserts that a stochastically continuous Markov semigroup is stable if it admits an invariant measure, is irreducible
and has the strong Feller property; see \cite{doob48}, \cite[Thm 4.4]{gerlach2014}, 
\cite{gerlach2012}, \cite{stettner1994} and \cite{seidler1997} for various versions and proofs of this result.
We recall that a semigroup on the space of measures is said to have the 
strong Feller property if its adjoint maps bounded measurable functions to continuous ones.

A necessary condition for stability of a semigroup is weak ergodicity, i.e.\ convergence
of the Ces\`aro averages in the weak topology induced by the bounded continuous functions.
In \cite{gerlach2013} M.\ Kunze and the author characterized 
ergodicity of semigroups on general norming dual pairs 
in the spirit of the classical mean ergodic theorem.
In particular for eventually strong Feller Markov semigroups it was shown that they are weakly ergodic if
the space of invariant measures separates the space of invariant continuous functions.

In the present article we prove that for eventually strong Feller semigroups 
weak ergodicity is already sufficient for stability, i.e.\ pointwise convergence of the semigroup
in the total variation norm.
In comparison to Doob's classical result, this Tauberian theorem even shows stability of not necessarily irreducible semigroups whose fixed space
is of arbitrary high dimension.

\medskip
In the following section 
we show that the square of every strong Feller operator is a kernel operator
in the sense that it belongs to the band generated by the finite rank operators.
Section \ref{sec:main} addresses Markov semigroups and their asymptotic behavior and 
contains the proof of our main result, Theorem \ref{thm:tauberian}.

\section{Strong Feller and Kernel operators}
\label{sec:kernelops}
Throughout, $\Omega$ denotes a Polish space and $\mathscr{B}(\Omega)$ its  Borel $\sigma$-algebra.
We denote by $\mathscr{M}(\Omega), B_b(\Omega)$ and $C_b(\Omega)$ the spaces of signed measures on 
$\mathscr{B}(\Omega)$, the space of bounded, Borel-measurable functions on $\Omega$ and the space 
of bounded continuous functions on $\Omega$, respectively.
We denote by $\applied{\; \cdot\;}{\; \cdot\;}$ the duality between $B_b(\Omega)$ and $\mathscr{M}(\Omega)$.

\begin{defn}
\label{def:transkernel}
	A \emph{Markovian transition kernel} on $\Omega$ is a map $k: \Omega\times\mathscr{B}(\Omega)\to \R_+$ such that
	\begin{enumerate}[(a)]
	\item $A \mapsto k(x, A)$ is a probability measure for every $x\in \Omega$ and
	\item $x \mapsto k(x,A)$ is a measurable function for every $A\in \cB(\Omega)$.
	\end{enumerate}
\end{defn}

To each Markovian transition kernel $k$, one can associate a positive operator $T \in \mathscr{L}(\mathscr{M}(\Omega))$ by setting
\begin{align}\label{eq.kernel}
(T\mu)(A)  \coloneqq \int_\Omega k(x,A)\, \dx\mu (x).
\end{align}
for all $\mu \in \cM(\Omega)$ and $A \in \cB(\Omega)$. 
The following lemma characterizes operators of this form.

\begin{lemma}
\label{lem:weaklycont}
For a positive operator $T \in \mathscr{L}(\mathscr{M}(\Omega))$ the following assertions are equivalent:
\begin{enumerate}[(i)]
\item There exists a Markovian transition kernel $k$ such that $T$ is given by \eqref{eq.kernel}.
\item The norm adjoint $T^*$ of $T$ leaves $B_b(\Omega)$ invariant and $T^* \mathds{1} = \mathds{1}$.
\item The operator $T$ is continuous in the $\sigma(\mathscr{M}(\Omega), B_b(\Omega))$-topology.
\end{enumerate}
\end{lemma}
\begin{proof}
This follows from Propositions 3.1 and 3.5 of \cite{kunze2011}.
\end{proof}

\begin{defn}
If an operator $T\in \cL(\cM(\Omega))$ satisfies the equivalent conditions from Lemma \ref{lem:weaklycont},
then $T$ is called \emph{Markovian} and we write $T'$ for the restriction of $T^*$ to $B_b(\Omega)$.

If a Markovian operator $T\in \cL(\cM(\Omega))$ even satisfies $T' f \in C_b(\Omega)$ for all $f\in B_b(\Omega)$, then $T$
is called \emph{strong Feller} and if, in addition, the family
	\[ \{ T'f : f \in B_b(\Omega),\; \lvert f\rvert \leq c\mathds{1} \} \]
	is equi-continuous for all $c> 0$, then $T$ is said to be \emph{ultra Feller}.
\end{defn}

It is well know that the product of two strong Feller operators is ultra Feller, see \cite[\S 1.5]{revuz1975}.

We recall that for two Riesz spaces $E$ and $F$ a linear operator from $E$ to $F$ is called regular if
it is the difference of two positive operators. If the Riesz space $F$ is order complete, the regular
operators from $E$ to $F$ form itself a order complete Riesz space. By $E^*_{\text{oc}}$ we denote the order continuous
linear functionals on $E$.

\begin{defn}
\label{def:kernelops}
Let $E$ and $F$ be Riesz spaces and $F$ be order complete. 
We denote by $E^*_{\text{oc}}\otimes F$ the space of order continuous finite rank operators from $E$ to $F$.
The elements of $(E^*_{\text{oc}}\otimes F)^{\bot\bot}$, the band generated by $E^*_{\text{oc}}\otimes F$ in the regular operators from $E$ to $F$,
are called \emph{kernel operators}.
\end{defn}

Since $\mathscr{M}(\Omega)$ is a $L$-space, its norm is order continuous. Therefore,
every bounded linear operator on $\mathscr{M}(\Omega)$ is regular and order continuous
and $\cM(\Omega)^*=\cM(\Omega)^*_{\text{oc}}$. Thus, 
$\mathscr{L}(\mathscr{M}(\Omega))$ is an order complete
Banach lattice with respect to the natural ordering, see \cite[Thm IV 1.5]{schaefer1974}.

We use the following 
characterization of kernel operators on $L^\infty$-spaces due to Bukhvalov:
\begin{thm}
\label{thm:bukhvalov}
Let $\mu$ and $\nu$ be finite measures on $(\Omega,\cB(\Omega))$ and $T$ a bounded
linear operator from $L^\infty(\Omega,\nu)$ to $L^\infty(\Omega,\mu)$.
Then $T$ is a kernel operator if and only if $\lim Tf_n =0$ $\mu$-almost everywhere
for each bounded sequence $(f_n)\subset L^\infty(\Omega,\nu)$ satisfying $\lim \lVert f_n \rVert_{L^1(\Omega,\nu)} =0$.
\end{thm}
\begin{proof}
This follows from Bukhvalov's theorem \cite[Thm 96.5]{zaanen1983} and the identification of 
concrete and abstract kernel operators \cite[Prop IV 9.8]{schaefer1974}.
\end{proof}

\begin{thm}
\label{thm:ultrafellerkernel}
	Let $T\in\cL(\cM(\Omega))$ be an ultra Feller operator.
	Then $T$ is a kernel operator.
\end{thm}
\begin{proof}
	Let $\mu \in \cM(\Omega)_+$ and $\nu \coloneqq T\mu$,
	then $T \{ \mu\}^{\bot\bot} \subset \{\nu\}^{\bot\bot}$.
	Let us denote by $T_\mu$ the restriction of $T$ to $\{\mu\}^{\bot\bot}$.
	By the Radon-Nikodym theorem, $\{\mu\}^{\bot\bot}$ and $\{\nu\}^{\bot\bot}$ 
	are isometrically isomorphic to $L^1(\Omega,\mu)$ and $L^1(\Omega,\nu)$.
	Thus, we may consider $T_\mu$ as an operator from $L^1(\Omega,\mu)$ to $L^1(\Omega,\nu)$ and we prove that 
	$T_\mu^* : L^\infty(\Omega,\nu)\to L^\infty(\Omega,\mu)$ is a kernel
	operator by applying Bukhvalov's theorem in the version of Theorem \ref{thm:bukhvalov}.
	It is easy to check that
	\[ T_\mu^* [f] = [T' f]\]
	for every $f\in B_b(\Omega)$, where $[f]$ denotes the equivalent class of $f$ in $L^\infty(\Omega,\mu)$.

	Let $(f_n) \subset  L^\infty(\Omega,\nu)$ be a bounded sequence such that
	$\lim \lVert f_n \rVert_{L^1(\Omega,\nu)} =0$.
	By choosing representatives we may assume that
	every $f_n$ is a bounded measurable function. Moreover, we may assume that each $f_n$ vanishes on  
	$\Omega\setminus \supp (\nu)$.  Then
	\[0= \lim \applied{f_n}{\nu}=\lim \applied{T_\mu^*f_n}{\mu} = \lim \applied{T_\mu'f_n}{\mu}.\]
	Let $\omega \in \Omega$ such that $(T_\mu'f_n)(\omega)$ does not converge to $0$. Then there exists $\eps>0$
	and a subsequence $(f_{n_k})$ of $(f_n)$ such that $T_\mu'f_{n_k}(\omega) \geq \eps$ for all $k\in\N$.
	By the ultra Feller property of $T_\mu$, the family
	\[ \{ T_\mu'f_{n_k} : k\in \N \} \]
	is equi-continuous. Therefore, we find an open neighborhood $U$ of $\omega$ such that $(T_\mu'f_{n_k})(s)\geq \eps/2$
	for all $s\in U$ and $k\in\N$. Now we conclude from
	\[ \mu(U)\frac{\eps}{2} \leq \int_\Omega T_\mu'f_{n_k} \dx\mu \to 0 \quad (k\to\infty)\]
	that $\mu(U)=0$ and hence $U \subset \Omega\setminus \supp(\mu)$.
	This proves that $(T_\mu' f_n)(\omega)$ converges to $0$ for all $\omega\in \supp(\mu)$ and hence almost everywhere.
	Thus, 
	it follows from Theorem \ref{thm:bukhvalov} that
	\[ T_\mu^*\in (L^\infty(\Omega,\nu)^*_{\text{oc}} \otimes L^\infty(\Omega,\mu))^{\bot\bot}.\]
	By \cite[Prop 1.4.15]{meyer1991}, the order continuous functionals on $L^\infty(\Omega,\nu)$ are precisely $L^1(\Omega,\nu)$.
	Thus, $T_\mu^*\in (L^1(\Omega,\nu) \otimes L^\infty(\Omega,\mu))^{\bot\bot}$.

	Now we prove that
	$T_\mu \in (L^\infty(\Omega,\mu)\otimes L^1(\Omega,\nu))^{\bot\bot}$. Let $0\leq S_\alpha\leq T_\mu^*$, $\alpha\in\Lambda$,
	be an upwards directed net and $R_\alpha \in L^1(\Omega,\nu)\otimes L^\infty(\Omega,\mu)$ such that $\sup S_\alpha=T_\mu$
	and $S_\alpha \leq R_\alpha$ for all $\alpha\in\Lambda$.
	Then $S_\alpha^* \leq R_\alpha^* \in L^\infty(\Omega,\mu)\otimes L^1(\Omega,\nu)$ for all $\alpha\in\Lambda$. 
	Since $L^1(\Omega,\nu)$ is an
	ideal in the dual of $L^\infty(\Omega,\nu)$, we obtain that 
	\[ {S_\alpha^*}\mid_{L^1(\Omega,\mu)} : L^1(\Omega,\mu) \to L^1(\Omega,\nu).\]
	Now it follows from
	\[ \sup \applied{(T_\mu-S^*_\alpha)f}{g} = \sup \applied{f}{(T_\mu^*-S_\alpha)g} = 0\]
	for all $f\in L^1(\Omega,\mu)$ and $g\in L^\infty(\Omega,\nu)$ that $T_\mu = \sup S^*_\alpha$
	and therefore 
	\[T\mid_{ L^1(\Omega,\mu)} = T_\mu \in (L^\infty(\Omega,\mu)\otimes L^1(\Omega,\nu))^{\bot\bot}.\]
	It follows that $TP_\mu \in (\cM(\Omega)^*\otimes \cM(\Omega))^{\bot\bot}$ for every $\mu\in \cM(\Omega)_+$
	where $P_\mu$ denotes the band projection onto $\{\mu\}^{\bot\bot}$. Thus,
	\[ T = \sup \{ TP_\mu : \mu \in \cM(\Omega)_+ \} \]
	belongs to $(\cM(\Omega)^*\otimes \cM(\Omega))^{\bot\bot}$ which completes the proof.
\end{proof}

\section{Stability of ergodic strong Feller semigroups}
\label{sec:main}

A \emph{Markovian semigroup} on $\cM(\Omega)$ is a family
$\cT=(T(t))_{t\geq 0} \subset \cL(\cM(\Omega))$ of Markovian 
operators on $\cM(\Omega)$ such that $T(t+s) = T(t)T(s)$ for all $t,s\geq 0$ and $T(0)=I$.
A Markovian semigroup is called 
\emph{stochastically continuous} if $t\mapsto \applied{T(t)\mu}{f}$ is
continuous for all $f\in C_b(\Omega)$ and $\mu \in \cM(\Omega)$.

Throughout, let $\cT=(T(t))_{t\geq 0}$ be a stochastically continuous Markovian semigroup.

It follows from \cite[Thm 6.2]{kunze2011} that $\cT$ is \emph{integrable} in the sense of \cite[Def 5.1]{kunze2011}.
In particular, by \cite[Thm 5.8]{kunze2011}, for every $t>0$ there exists a Markovian operator
$A_t \in \cL(\cM(\Omega))$ satisfying
	\[ \applied{A_t \mu}{f} = \frac{1}{t} \int_0^t \applied{T(s)\mu}{f} \dx s\]
	for all $\mu\in \cM(\Omega)$ and $f\in B_b(\Omega)$.
We call the semigroup $\cT$ $B_b$-\emph{ergodic}
	if $\lim_{t\to\infty} A_t \mu$ exists in the $\sigma(\cM(\Omega),B_b(\Omega))$-topology 
	for all $\mu\in\cM(\Omega)$.

The following proposition ensures that for every initial distribution the part on the disjoint
complement of $\fix(\cT)$ converges to zero 
if $\cT$ is $B_b$-ergodic and eventually strong Feller.

\begin{prop}
	\label{prop:stabilityonfixbot}
	Let $P$ denote the band projection onto $\fix(\cT)^{\bot}$. 
	If $\cT$ is $B_b$-ergodic and $T(t_0)$ is strong Feller for some $t_0>0$,
	then \[\lim_{t\to\infty} PT(t)\mu = 0\] for all $\mu \in \cM(\Omega)$.
\end{prop}
\begin{proof}
	First note that, since $\fix(\cT)^{\bot\bot}$ is $\cT$-invariant, $R(t) \coloneqq PT(t)$ defines a semigroup.
	Obviously, every operator $R(t)$ is positive and contractive.
	Fix $\mu \in \cM(\Omega)_+$ and let
	\[ \alpha \coloneqq \lim_{t\to\infty} \lVert PT(t) \mu \rVert = \inf_{t\geq 0} \lVert R(t)\mu  \rVert. \]
	We pick $t_1>0$ such that $\nu \coloneqq PT(t_1)\mu$ satisfies $\lVert \nu \rVert < \alpha+\frac{\alpha}{2}$.
	Since $\cT$ is $B_b$-ergodic,
	$\lim A_t \nu \eqqcolon \tilde\nu$ exists with respect to the $\sigma(\cM(\Omega),B_b(\Omega))$-topology
	and $\tilde \nu \in \fix(\cT)$ by \cite[Lem 4.5]{gerlach2013}.
	In particular,
	\[\applied{\tilde\nu}{\mathds{1}} = \lim_{t\to\infty} \frac{1}{t} \int_0^t \applied{T(s)\nu}{\mathds{1}}\dx s = \lVert\nu\rVert \geq \alpha.\]
	Let $t\geq t_0$. As $PT(2t)\nu$ and $\nu$ are disjoint, there exists a Borel set $B \subset \Omega$ such that 
	\[(PT(2t)\nu)(B) = \tilde \nu (\Omega\setminus B)=0.\]
	Since $\lVert T(2t)\nu \rVert \leq \lVert \nu\rVert < \alpha+\frac{\alpha}{2}$ and
	\[ \lVert PT(2t)\nu\rVert = \lVert R(t_1+2t)\mu \rVert \geq \alpha\]
	it follows from the additivity of the total variation norm that
	$\lVert (I-P)T(2t)\nu \rVert < \frac{\alpha}{2}$.
	Hence, $(T(2t)\nu)(B) < \frac{\alpha}{2}$. Let $f\coloneqq T'(t)\mathds{1}_B$ and $g\coloneqq T'(t)\mathds{1}_{\Omega\setminus B}$.	
	Since $T(t)=T(t-t_0)T(t_0)$ is strong Feller and Markovian, $f,g \in C_b(\Omega)_+$ and $f+g = \mathds{1}$.
	It follows from $\applied{\tilde \nu}{g}=0$ that 
	\[ A \coloneqq \supp\tilde\nu \subset \{ g=0 \} = \{ f=1\},\]
	i.e. $\mathds{1}_A \leq f$. Thus,
	\[ \applied{T(t)\nu}{\mathds{1}_A} \leq \applied{T(t)\nu}{f} = \applied{T(2t)\nu}{\mathds{1}_B} < \frac{\alpha}{2}.\]
	Since $t\geq t_0$ was arbitrary, we conclude that
	$\applied{T(t)\nu}{\mathds{1}_A} < \frac{\alpha}{2}$ for all $t\geq t_0$ and hence
	\[ \alpha = \applied{\tilde\nu}{\mathds{1}_A} = \lim_{t\to\infty} \frac{1}{t} \int_0^t \applied{T(s)\nu}{\mathds{1}_A} \leq \frac{\alpha}{2}.\]
	Thus, $\alpha=0$.
\end{proof}

\begin{rem}
	The assumption that $\cT$ is eventually strong Feller cannot be dropped in Proposition \ref{prop:stabilityonfixbot}. 
	A counterexample is given by the rotation group on the Borel measures on the unit circle.

	Note that the situation is different for time-discrete semigroups.
	If $T$ is a positive and mean ergodic contraction on an $L$-space $E$ and $P$ is the band projection onto $\fix(T)^\bot$, 
	then it is possible to prove that 
	\[\lim_{n \to\infty} \lVert P T^n x \rVert = 0\]
	 for all $x\in E$.
\end{rem}

Our main tool for the proof of the desired Tauberian theorem 
is the following result from \cite[Thm 4.2]{gerlach2012b},
a generalized version of \cite[Kor 3.11]{greiner1982}.
We recall that a family $\cS = (S(t))_{t\geq 0} \subset \cL(E)$ of positive operators on a Banach lattice $E$
is called a \emph{positive strongly continuous semigroup} if $S(t)S(s) = S(t+s)$ for all $t,s\geq 0$, $S(0)=I$ and
the mapping $t\mapsto S(t)x$ is continuous for all $x\in E$.
A positive strongly continuous semigroup $\cS = (S(t))_{t\geq 0}$ is called \emph{irreducible} if $\{0\}$ and $E$ 
are the only closed ideals in $E$ that are invariant under the action of every operator $S(t)$.

\begin{thm}[Greiner]
	\label{thm:greiner}
	Let $\cS = (S(t))_{t\geq 0} \subset \cL(E)$ be a positive, bounded, irreducible and strongly continuous semigroup on a Banach lattice $E$
	with order continuous norm such that $\fix(\cS) \not= \{0\}$.
	If $S(t_0)\in (E^*_{\text{oc}}\otimes E)^{\bot\bot}$ for some $t_0>0$, then
	there exists a positive $z^* \in \fix(\cS^*)$ and a positive $z \in \fix(\cS)$ of $E_+$
	such that
	\[ \lim_{t\to\infty} S(t)x = \applied{z^*}{x} z\]
	for all $x\in E$.
\end{thm}

The following theorem shows that, if the semigroup $\cT$ contains a kernel operator, 
every principal band $\{\mu\}^{\bot\bot}$ spanned by an invariant measure $\mu$ can be decomposed into
countably many invariant bands such that the restriction of $\cT$ to each of them is irreducible.

\begin{thm}
	\label{thm:irreddecomp}
	Let $\mu\in \cM(\Omega)$ be a positive $\cT$-invariant measure. If $T(t_0)$ is a kernel operator, 
	then there exist at most countably many disjoint $\cT$-invariant measures $\{ \mu_n \} \subset \cM(\Omega)_+$
	such that $\mu =  \mu_1+\mu_2+\dots$ and the restriction of $\cT$ to each $\{\mu_n\}^{\bot\bot}$ is 
	irreducible.

\end{thm}
\begin{proof}
	By the Radon-Nikodym theorem we may identify the band $\{\mu\}^{\bot\bot}$ with
	$L^1(\Omega,\mu)$ and thus consider $\cT$ as a contractive semigroup on $L^1(\Omega,\mu)$.
	Since the measure $\mu$ is $\cT$-invariant and corresponds to $\mathds{1} \in L^1(\Omega,\mu)$,
	$T(t)\mathds{1}_B \leq \mathds{1}$ for all $B\in \cB(\Omega)$ and $t\geq 0$.

	In this proof, 
	we call a Borel set $B\subset \cB(\Omega)$ \emph{invariant} if $T(t)\mathds{1}_B \leq \mathds{1}_B$ almost
	everywhere for all $t\geq 0$ and \emph{irreducible} 
	if for every invariant Borel set $A\subset B$ we have $\mu(A)=0$ or $\mu(A)=\mu(B)$.

	With this identification and notation, we have to find at most
	countable many disjoint invariant and irreducible Borel sets $B_1,B_2,\dots$ such that
	$B_1 \cup B_2 \cup \dots = \Omega$.

	First, we show that a Borel set $B$ is invariant if and only if $\Omega\setminus B$
	is invariant if and only if $\mathds{1}_B, \mathds{1}_{\Omega\setminus B}\in \fix(\cT)$.
	Let $B\in \cB(\Omega)$ be invariant. Then, for every $t\geq 0$,
	\[ T(t)\mathds{1}_{\Omega\setminus B} = T(t)\mathds{1} - T(t)\mathds{1}_B \geq \mathds{1} - \mathds{1}_B = \mathds{1}_{\Omega\setminus B} \quad \mu\text{-almost everywhere.}\]
	Since $T(t)$ is contractive, $\mathds{1}_{\Omega\setminus B}$ is a fixed point of $T(t)$ and so is $\mathds{1}_B$.

	Next, we prove the existence of an irreducible Borel set of positive measure.
	Aiming for a contradiction, we assume that $\mu(B)=0$ for every irreducible $B\in \cB(\Omega)$.
	For $n\in\N$ define
	\[ \cA_n \coloneqq \left\{ \mathds{1}_A : A\subset \cB(\Omega) \text{ is invariant and } \mu(A)\leq \frac{1}{n} \right\}\]
	and let $B_n\in \cB(\Omega)$ such that
	\[ \mathds{1}_{B_n} = \bigvee_{A\in \cA_n} \mathds{1}_A \quad \mu\text{-almost everywhere,} \]
	where the supremum is taken in the order complete lattice $L^\infty(\Omega,\mu)$.
	Then $B_n$, hence by the above also $\Omega\setminus B_n$, is invariant. 
	Since we assumed every irreducible set to be a null-set
	and $\mu(\tilde B)>\frac{1}{n}$ for every measurable invariant subset $\tilde B \subset \Omega\setminus B_n$, we conclude that
	$\mu(\Omega\setminus B_n)=0$.  Therefore, $\mathds{1}_{B_n} = \mathds{1}$ almost everywhere.

	Let $\cD_n$ be a maximal disjoint system in $\cA_n$. 
	By the countable sup property of $L^\infty(\Omega,\mu)$, see \cite[Thm 8.22]{aliprantis2006}, we obtain
	the existence of a countable subset $(\mathds{1}_{A_{k,n}})_{k\in\N} \subset \cD_n$
	with $\sup_{k\in\N} \mathds{1}_{A_{k,n}} = \mathds{1}$. Since the functions $\{\mathds{1}_{A_{k,n}} : k\in\N\}$
	are pairwise disjoint, it follows that $\lim_{k\to\infty} \lVert \mathds{1}_{A_{k,n}} \rVert_{L^1}=0$
	for every $n\in\N$.
	By ordering the sets $A_{k,n}$ decreasing in measure, 
	we obtain a single sequence $(A_n)\subset \cB(\Omega)$ that contains
	every set $A_{k,n}$ and satisfies $\lim \lVert \mathds{1}_{A_n} \rVert_{L^1} = 0$.
	Now it follows from Theorem \ref{thm:bukhvalov} that $\lim T(t_0)\mathds{1}_{A_n} = 0$ almost everywhere
	in contradiction to 
	\[ \sup_{n\geq k} T(t_0) \mathds{1}_{A_n} = \sup_{n\geq k} \mathds{1}_{A_n} = \mathds{1} \]
	for all $k\in\N$. Thus, there exists an irreducible set $B\in\cB(\Omega)$ with $\mu(B)>0$.
	Moreover, the same argument shows that every invariant set $\tilde \Omega \in\cB(\Omega)$ of positive measure contains 
	an irreducible Borel set of positive measure.

	Let $\cD \coloneqq \{ D \in \cB(\Omega) : D \text{ is irreducible} \}$.
	Then $\sup\{ \mathds{1}_D : D \in \cD \} = \mathds{1}$. By the countable sup property, we find a sequence $(D_n)\subset \cD$
	with $\sup \mathds{1}_{D_n} = \mathds{1}$. This proves the claim.
\end{proof}

Let us note that,
by applying a general version of Bukhvalov's theorem proven in \cite{grobler1980},
Theorem \ref{thm:irreddecomp} can be generalized to contractive semigroups containing a kernel operator
on an arbitrary Banach lattice whose norm is strictly monotone and order continuous.

Combining Greiner's Theorem \ref{thm:greiner} and the irreducible decomposition of Theorem \ref{thm:irreddecomp},
we obtain stability of $\cT$ on the band spanned by its fixed space.

\begin{prop}
	\label{prop:convfixed}
	If $T(t_0)$ is a kernel operator for some $t_0>0$, then
	$\lim_{t\to\infty} T(t) \mu$ exists for all $\mu\in \fix(\cT)^{\bot\bot}$.
\end{prop}
\begin{proof}
	Since every $T(t)$ is a contraction and the total variation norm is strictly monotone on the positive cone $\cM(\Omega)_+$,
	for all $\mu\in \fix(\cT)$ it follows from
	\[ \lvert \mu\rvert = \lvert T(t)\mu\rvert \leq T(t) \lvert \mu\rvert \]
	that $T(t)\lvert \mu\rvert = \lvert \mu\rvert$.
	Hence, $\fix(\cT)$ is a sublattice. 

	Now let $\mu\in \fix(\cT)^{\bot\bot}_+$ and denote by $P$ the band projection onto $\fix(\cT)^{\bot}$.
	Let $\cD$ be a maximal disjoint system in $\fix(\cT)_+$. 
	Since the total variation norm on $\cM(\Omega)$ is a $L$-norm, i.e. it is additive on the positive cone $\cM(\Omega)_+$,
	there exists an
	at most countable subset $\mathscr{C} \subset \cD$ such that $\mu \in \cC^{\bot\bot}$.
	In fact, for $\zeta\in \cD$ let $P_\zeta$ denote the band projection onto $\{\zeta\}^{\bot\bot}$. 
	Then for every $m\in\N$
	there exist only finitely many $\zeta\in \cD$ such that $\lVert P_\zeta \mu \rVert \geq \frac{1}{m}$. This implies that there
	are at most countably many $\zeta\in \cD$ such that $P_\zeta \mu >0$. 
	Let $\mathscr{C} \coloneqq (\zeta_k) \coloneqq \{ \zeta \in \cD$ and $(\mu_k) \coloneqq (P_\zeta \mu)_{\zeta \in \cC}$. 

	Since $\mu$ is a fixed point of $\cT$, the band $\{\mu \}^{\bot\bot}$ is $\cT$-invariant. 
	By  Theorem \ref{thm:irreddecomp}, we may assume that 
	the restriction of $\cT$ to $\{\mu\}^{\bot\bot}$ is irreducible. Moreover,
	since $\cT$ is stochastically continuous, this restriction
	is strongly continuous by \cite[Thm 4.6]{hille2009}.
	Thus,
	for each $k\in \N$, the limit $\nu_k \coloneqq \lim_{t\to\infty} T(t)\mu_k \in \{ \zeta_k \}^{\bot\bot}$ exists
	by Theorem \ref{thm:greiner} applied to the Banach lattice $\{\zeta_k\}^{\bot\bot}$.

	Next, we show that $\tau_n \coloneqq \nu_1 + \dots + \nu_n$ is a  Cauchy sequence.
	For a given $\eps>0$ choose $n\in\N$ such that $\sum_{k=n+1}^\infty \lVert \mu_k \rVert < \eps$. 
	Then
	\[\lVert \tau_n - \tau_m \rVert =  \sum_{k=n+1}^m \lVert \nu_k \rVert  \leq \sum_{k=n+1}^\infty \lVert \mu_k \rVert < \eps\]
	for all $m >n$. Therefore,
	$\tau \coloneqq \lim \tau_m \in \fix(\cT)^{\bot\bot}$ exists. We prove that $\lim T(t)\mu = \tau$.
	Let $\eps>0$ and choose $n\in\N$ such that $\sum_{k=n+1}^\infty \lVert \mu_k\rVert < \eps$.
	Since $T(t)\mu_k$ converges to $\nu_k$ we find $s>0$ such that 
	$\lVert T(t)\mu_k - \nu_k\rVert < \eps/n$ for all $t\geq s$ and all $1\leq k\leq n$.
	Finally, we obtain that
	\begin{align*}
		\lVert T(t)\mu - \tau\rVert &\leq \sum_{k=1}^\infty \lVert T(t)\mu_k - \nu_k \rVert \\
		&\leq \sum_{k=1}^n \lVert T(t)\mu_k - \nu_k\rVert + \sum_{k=n+1}^\infty 2 \lVert \mu_k\rVert  < n\cdot \frac{\eps}{n} + 2\eps
	\end{align*}
	for all $t\geq s$. This shows that $\lim_{t\to\infty} T(t)\mu = \tau$.
\end{proof}

Let us remark that, using a generalized version of Theorem \ref{thm:irreddecomp},
Proposition \ref{prop:convfixed} remains true for every positive and contractive semigroup $\cT=(T(t))_{t\geq 0}$
on a Banach lattice with strictly monotone and order continuous norm such that the restriction of $\cT$ to $\{x\}^{\bot\bot}$ is
strongly continuous for every $x\in \fix(\cT)^{\bot\bot}$.

Now we prove our main result.

\begin{thm}
	\label{thm:tauberian}
	If $\cT$ is $B_b$-ergodic and $T(t_0)$ is strong Feller for some $t_0>0$, 
	then $\lim_{t\to\infty} T(t)\mu$ exists for all $\mu\in\cM(\Omega)$.
\end{thm}
\begin{proof}
	Let $\mu \in \cM(\Omega)_+$ and denote by $P$ the band projection onto $\fix(\cT)^{\bot}$.
	By Proposition \ref{prop:stabilityonfixbot}, there exists an increasing sequence $t_n >0$ such that
	$\lVert PT(t_n)\mu \rVert < \frac{1}{n}$ for all $n\in\N$.
	Define 
	\[\mu_n \coloneqq (I-P)T(t_n)\mu \in \fix(\cT)^{\bot \bot}.\]
	It follows from \cite[\S 1.5]{revuz1975} that $T(2t_0)$ is ultra Feller and therefore a kernel
	operator by Theorem \ref{thm:ultrafellerkernel}. Therefore, by
	Proposition \ref{prop:convfixed}, $\nu_n \coloneqq \lim_{t\to\infty} T(t)\mu_n$ exists in $\fix(\mathscr{S})^{\bot\bot}$
	for every $n\in\N$.
	Hence, there exists an increasing sequence $s_n > 0$ such that
	$\lVert T(t)\mu_n -\nu_n\rVert < \frac{1}{n}$ for all $n\in\N$ and $t\geq s_n$.
	This implies that for every $n\in\N$
	\begin{align*}
	\lVert T(t+t_n)\mu - \nu_n \rVert &\leq \lVert T(t) (I-P) T(t_n)\mu - \nu_n \rVert + \lVert T(t) P T(t_n) \mu\rVert\\
	&\leq \lVert T(t)\mu_n -\nu_n \rVert + \lVert PT(t_n)\mu \rVert < \frac{2}{n}
	\end{align*}
	for all $t \geq s_n$.
	Since
	\begin{align*}
	\lVert \nu_n - \nu_m \rVert &\leq \lVert \nu_n - T(s_m+t_m)\mu\rVert + \lVert T(s_m +t_m)\mu -\nu_m\rVert <\frac{2}{n} + \frac{2}{m}
	\end{align*}
	for all $m\geq n$, $(\nu_n)$ is a Cauchy sequence. Let $\nu\coloneqq \lim \nu_n \in \fix(\mathscr{S})^{\bot\bot}$. 
	Then for every $\eps>0$ there exists $n\in\N$
	such that
	\begin{align*}
		\lVert T(t)\mu - \nu\rVert &\leq \lVert T(t)\mu - \nu_n \rVert + \lVert \nu_n -\nu \rVert < \eps
	\end{align*}
	for all $t\geq t_n+s_n$ which proves the claim.
\end{proof}

Making use of the characterization of weak ergodicity in \cite[Thm 5.7]{gerlach2013}, we obtain the following Corollary.

\begin{cor}
	If $T(t_0)$ is strong Feller for some $t_0>0$, then
	the following are equivalent
\begin{enumerate}[(i)]
\item $\fix(\cT)$ separates $\fix(\cT') \coloneqq \{ f\in C_b(\Omega) : T'(t)f=f \text{ for all } t\geq 0\}$.
\item The semigroup $\cT$ is weakly ergodic in the sense that $\lim_{t\to\infty} A_t \mu$ exists in the $\sigma(\cM(\Omega),C_b(\Omega))$-topology 
	for all $\mu\in\cM(\Omega)$.
\item The semigroup $\cT$ is $B_b$-ergodic.
\item $\lim_{t\to\infty} T(t)\mu$ exists for each $\mu\in\cM(\Omega)$.
\end{enumerate}
\end{cor}
\begin{proof}
	Let us assume (i) and pick $\mu \in \cM(\Omega)$. 
	It follows from \cite[Thm 5.7]{gerlach2013} that there exists $\tilde \mu \in \fix(\cT)$ such that
	\[ \lim \applied{A_t \mu - \tilde \mu}{f} = 0\]
	for all $f\in C_b(\Omega)$, i.e.\ assertion (ii) holds.

	As explained in \cite[Ex 3.6]{gerlach2013}, one has that
	\[ \lim_{t\to\infty} \lVert (T(t_0) - I)A_t \mu \rVert =0. \]
	Since $T(t_0)$ is strong Feller, assertion (ii) implies that
	\[ \lim_{t\to\infty} \applied{A_t \mu-\tilde \mu}{f} = \lim_{t\to\infty} \applied{A_t \mu - T(t_0)A_t \mu}{f} + \applied{A_t \mu - \tilde \mu}{T'(t_0)f} =0\]
	for all $f\in B_b(\Omega)$, i.e.\  $\cT$ is $B_b$-ergodic.
	
	Theorem \ref{thm:tauberian} yields that (iii) implies (iv).
	In order to prove that (i) follows from (iv), we assume that $\lim T(t)\mu$ exists for each $\mu\in \cM(\Omega)$.
	For $f\in \fix(\cT')$ choose $\mu\in \cM(\Omega)$ such that $\applied{\mu}{f} \eqqcolon \alpha \neq 0$.
	Let $\tilde\mu \coloneqq \lim T(t)\mu \in \fix(\cT)$. Then 
	\[ \applied{\tilde \mu}{f} = \lim_{t\to\infty} \applied{T(t)\mu}{f} = \alpha \neq 0\]
	which shows that $\fix(\cT)$ separates $\fix(\cT')$.
\end{proof}

\bibliographystyle{abbrv}
\bibliography{analysis}

\end{document}